\input ./style/arxiv-vmsta.cfg
\documentclass[numbers,compress,v1.0.1]{vmsta}

\volume{2}
\issue{3}
\pubyear{2015}
\firstpage{203}
\lastpage{218}
\doi{10.15559/15-VMSTA32}

\startlocaldefs
\newcommand{\rrvert}{\vert}
\newcommand{\llvert}{\vert}
\newcommand{\RR}{\mathbb{R}}
\newcommand{\eps}{\varepsilon}
\newcommand{\F}{\mathcal{F}}
\newcommand{\pr}{\mathsf{P}}
\newcommand{\wnu}{\widetilde{\nu}}

\newtheorem{theorem}{Theorem}[section]

\theoremstyle{definition}

\newtheorem{remark}{Remark}[section]
\newtheorem{example}{Example}[section]

\allowdisplaybreaks
\endlocaldefs

\begin{document}
\begin{frontmatter}

\title{Convergence of hitting times for jump-diffusion processes}

\author{\inits{G.}\fnm{Georgiy}\snm{Shevchenko}\corref{}}\email
{zhora@univ.kiev.ua}
\address{Taras Shevchenko National University of Kyiv, Mechanics and
Mathematics Faculty, Volodymyrska 64, 01601 Kyiv, Ukraine}

\markboth{G. Shevchenko}{Convergence of hitting times for
jump-diffusion processes}

\begin{abstract}
We investigate the convergence of hitting times for jump-diffusion
processes. Spe\-cifically, we study a sequence of stochastic differential
equations with jumps. Under reasonable assumptions, we establish the
convergence of solutions to the equations and of the moments when the
solutions hit certain sets.
\end{abstract}


\received{16 February 2015}
\revised{7 September 2015}
\accepted{7 September 2015}
\publishedonline{23 September 2015}
\end{frontmatter}

\section{Introduction}\label{sec1}

In this article, we consider a sequence of stochastic differential
equations with jumps
\begin{align*}
X^n(t) &= X^n(0) + \int_0^t
a^n \bigl(s,X^n(s) \bigr) ds + \int_0^t
b^n \bigl(s,X^n(s) \bigr)dW(s)
\\
&\quad+ \int_{0}^{t}\int_{\RR^m}c^n
\bigl(s,X^n(s-),\theta \bigr)\wnu(d\theta,ds), \quad t\ge0, \ n\ge0.
\end{align*}
Here $W$ is a standard Wiener process, $\wnu$ is a compensated Poisson
random measure, and $X^n(0)$ is nonrandom (see Section~\ref{sec:prelim}
for precise assumptions). Assuming that $a^n\to a^0$,
$b^n\to b^0$, $c^n\to c^0$, and $X^n(0)\to X^0(0)$ as $n\to\infty$ in
an~appropriate sense, we are interested in convergence of hitting times
$\tau^n \to\tau^0$, $n\to\infty$, where
\[
\tau^n = \inf \bigl\{t\ge0:\varphi^n
\bigl(t,X^n(t) \bigr)\ge 0 \bigr\}
\]
is the first time when the process $X^n$ hits the set $\mathcal G^n_t =
 \{x:\varphi^n(t,x)\ge0 \}$.

The study is motivated by the following observation. Jump-diffusion
processes are commonly used to model prices of financial assets. When
the parameters of a jump-diffusion process are estimated with the help
of statistical methods, there is an estimation error. Thus, it is
natural to investigate whether the optimal exercise strategies are
close for two jump-diffusion processes with close parameters. Moreover,
we should study particular hitting times since, in the Markovian
setting, the optimal stopping time is the hitting time of the optimal
stopping set.

There is a lot of literature devoted to jump-diffusion processes and
their applications in finance. The book \cite{cont-tankov} gives an
extensive list of references on the subject. The convergence of
stopping times for diffusion and jump-diffusion processes was studied
in \cite{tomashyk-mishura,moroz-tomashyk,tomashyk-shevchenko}. All
these papers are devoted to the one-dimensional case, and the
techniques are different from ours. Here we generalize these results to
the multidimensional case and also relax the assumptions on the
convergence of coefficients. As an auxiliary result of independent
interest, we prove the convergence of solutions under very mild
assumptions on the convergence of coefficients.

\section{Preliminaries and notation}\label{sec:prelim}
Let $(\varOmega,\F,\mathbf{F},\pr)$ be a standard stochastic basis
with filtration $\mathbf{F}=\{\F_t,t\geq0\}$ satisfying the usual
assumptions. Let $\{W(t)=(W_1(t),\dots,W_k(t)),t\geq0\}$ be a
standard Wiener process in $\RR^k$, and $\nu(d\theta,dt)$ be a
Poisson random measure on $\RR^m \times[0,\infty)$. We assume that
$W$ and $\nu$ are compatible with the filtration~$\mathbf{F}$,
that~is, for any $t>s\ge0$ and any $A\in\mathcal B(\mathbb R^m)$ and
$B\in\mathcal B([s,t])$, the increment $W(t)-W(s)$ and the value $\nu
(A\times B)$ are $F_t$-measurable and independent of~$\F_s$.

Assume in addition that $\nu(d\theta,dt)$ is homogeneous, that is,
for all $A\in\mathcal B(\mathbb R^m)$ and $B\in\mathcal B([0,\infty))$,
$\mathsf{E} [\nu(A\times B) ] = \mu(A)\lambda(B)$,
where $\lambda$ is the Lebesgue measure, $\mu$ is a $\sigma$-finite
measure on $\RR^m$ having no atom at zero. Denote by $\wnu$ the
corresponding compensated measure, that is,\ $\wnu(A\times B) = \nu
(A\times B) - \mu(A)\lambda(B)$ for all $A\in\mathcal B(\mathbb
R^m), B\in\mathcal B([0,\infty))$.

For each integer $n\ge0$, consider a stochastic differential equation
in $\RR^d$
\begin{align}
\label{sde-coord} %
X^n_i(t) &=
X^n_i(0) + \int_0^t
a^n_i \bigl(s,X^n(s) \bigr) ds + \sum
_{j=1}^{k}\int_0^t
b^n_{ij} \bigl(s,X^n(s) \bigr)dW_j(s)
\notag
\\
&\quad+ \int_{0}^{t}\int_{\RR^m}c^n_i
\bigl(s,X^n(s-),\theta \bigr)\wnu(d\theta,ds), \quad t\ge0, \ i=1,\dots,d.
\end{align}
In this equation, the initial condition $X^n(0)\in\RR^d$ is
nonrandom, and the coefficients $a^n_i,b^n_{ij}\colon[0,\infty)\times
\RR^d\to\RR$, $c^n_{i}\colon[0,\infty)\times\RR^d\times\RR
^m\to\RR$, $i=1,\dots,d$, $j=1,\dots,k$, are nonrandom and measurable.

In what follows, we abbreviate Eq.~\eqref{sde-coord} as
\begin{align}
\label{sde} %
X^n(t) &= X^n(0)+ \int
_0^t a^n \bigl(s,X^n(s)
\bigr) ds + \int_0^t b^n
\bigl(s,X^n(s) \bigr)dW(s)\notag
\\
&\quad+ \int_{0}^{t}\int_{\RR^m}c^n
\bigl(s,X^n(s-),\theta \bigr)\wnu(d\theta,ds), \quad t\ge0.
\end{align}

For the rest of the article, we adhere to the following notation. By
$\llvert\cdot\rrvert$ we denote the absolute value of a number, the
norm of a
vector, or the operator norm of a matrix, and by $(x,y)$ the scalar
product of vectors $x$ and $y$; $B_k(r) =  \{x\in\RR^k: \llvert
x\rrvert\le r \}$. The symbol $C$ means a generic constant whose value
is not
important and may change from line to line; a constant dependent on
parameters $a,b,c,\dots$ will be denoted by $C_{a,b,c,\dots}$.

The following assumptions guarantee that Eq.~\eqref{sde} has a unique
strong solution.
\begin{itemize}
\item[(A1)] For all $n\ge0$, $T> 0$, $t\in[0,T]$, $x\in\RR^d$,
\[
\bigl\llvert a^n(t,x) \bigr\rrvert^2 + \bigl\llvert
b^n(t,x) \bigr\rrvert^2 + \int_{\RR^m}
\bigl\llvert c^n(t,x,\theta) \bigr\rrvert^2 \mu(d\theta)
\le C_T \bigl(1 + \llvert x\rrvert^2 \bigr).
\]
\item[(A2)] For all $n\ge0$, $T\ge0$, $t\in[0,T]$, $R>0$, and
$x,y\in B_d(R)$
\begin{align*}
&\bigl\llvert a^n(t,x)-a^n(t,y) \bigr
\rrvert^2 + \bigl\llvert b^n(t,x)-b^n(t,y)
\bigr\rrvert^2
\\
&\qquad + \int_{\RR^m} \bigl\llvert c^n(t,x,
\theta)-c^n(t,y,\theta) \bigr\rrvert^2\mu(d\theta)\le
C_{T,R}\llvert x-y\rrvert^2.
\end{align*}
\end{itemize}

Moreover, under these assumptions, for any $T\ge0$, we have the
following estimate:
%
\begin{equation}
\label{momentbounded} \mathsf{E} \Bigl[\sup_{t\in[0,T]} \bigl\llvert
X^n(t) \bigr\rrvert^2 \Bigr] \le C_{T}
\bigl(1+ \bigl\llvert X^n(0) \bigr\rrvert^2 \bigr)
\end{equation}
(see, e.g.,\ \cite[Section~3.1]{situ}). From this estimate it is easy to
see from Eq.~ \eqref{sde} that for all $t,s\in[0,T]$,
%
\begin{equation}
\label{sol-cont} \mathsf{E} \bigl[ \bigl\llvert X^n(t)-X^n(s)
\bigr\rrvert^2 \bigr]\le C_{T} \bigl(1+ \bigl\llvert
X^n(0) \bigr\rrvert^2 \bigr)\llvert t-s\rrvert.
\end{equation}

Now we state the assumptions on the convergence of coefficients of
\eqref{sde}.
\begin{itemize}
\item[(C1)] For all $t\ge0$ and $x\in\RR^d$,
\begin{align*}
&a^n(t,x)\to a^0(t,x), \qquad b^n(t,x) \to
b^0(t,x),
\\
&\int_{\RR^m} \bigl\llvert c^n(t,x,
\theta)-c^0(t,x, \theta) \bigr\rrvert^2\mu(d\theta)\to0,
\quad n\to\infty.
\end{align*}
\item[(C2)] $X^n(0)\to X^0(0)$, $n\to\infty$.
\end{itemize}

\section{Convergence of solutions to stochastic differential equations with
jumps}\label{sec3}
First, we establish a result on convergence of solutions to stochastic
differential equations.
\begin{theorem}\label{ucp-thm}
Let the coefficients of Eq.~(\ref{sde}) satisfy assumptions
{(A1), (A2), (C1)}, and~(C2).
Then, for any $T>0$, we have the convergence in probability
\[
\sup_{t\in[0,T]} \bigl\llvert X^n(t)-X^0(t)
\bigr\rrvert\overset{ \pr} {\longrightarrow}0, \quad n\to\infty.
\]
If additionally the constant in assumption \textup{(A2)} is
independent of $R$, then for any $T>0$,
\[
\mathsf{E} \Bigl[\sup_{t\in[0,T]} \bigl\llvert X^n(t)-X^0(t)
\bigr\rrvert^2 \Bigr] \to0,\quad n\to\infty.
\]
\end{theorem}
\begin{proof}

Denote $\varDelta^n(t)= \sup_{s\in[0,t]} \llvert
X^n(t)-X^0(t)\rrvert$,
$a^{n,m}_s= a^n(s,X^m(s))$, $b^{n,m}_s= b^n(s,X^m(s))$,
$c^{n,m}_s(\theta)= c^n(s,X^m(s-),\theta)$,
\begin{align*}
&I^{n}_a(t) = \int_0^t
a^{n,n}_s ds,\qquad I^{n}_b(t) =
\int_0^t b^{n,n}_s dW(s),
\\
&I^{n}_c(t) = \int_0^t
\int_{\RR^m} c^{n,n}_s(\theta)\tilde\nu(d
\theta,ds).
\end{align*}
It is easy to see that $I^{n}_b$ and $I^{n}_c$ are martingales.

Write
\begin{align*}
\varDelta^n(t)^2 &\le C \Bigl( \bigl\llvert
X^n(0)-X^0(0) \bigr\rrvert^2+ \sup
_{s\in[0,t]} \bigl\llvert I_a^{n}(s)-I_a^{0}(s)
\bigr\rrvert^2
\\
&\quad+ \sup_{s\in[0,t]} \bigl\llvert I_b^{n}(s)-I_b^{0}(s)
\bigr\rrvert^2 + \sup_{s\in[0,t]} \bigl\llvert
I_c^{n}(s)-I_c^{0}(s) \bigr
\rrvert^2 \Bigr).
\end{align*}
For $N\ge1$, define
\[
\sigma^n_N = \inf \bigl\{t\ge0: \bigl\llvert
X^0(t) \bigr\rrvert\vee \bigl\llvert X^n(t) \bigr\rrvert
\ge N \bigr\}
\]
and denote $1_t = \mathbf{1}_{t\le\sigma^n_N}$.
Then
\begin{align*}
\mathsf{E} \bigl[\varDelta^n(t)^2 1_t\bigr]&\le\mathsf{E} \bigl[\varDelta^n \bigl(t\wedge\sigma^n_N \bigr)^2 \bigr] \\
&\le C \biggl(\bigl\llvert X^n(0)-X^0(0) \bigr\rrvert^2 + \sum_{x\in\{a,b,c\}}\mathsf{E} \Bigl[\sup_{s\in[0,t\wedge\sigma ^n_N]} \bigl\llvert I_x^{n}(s)-I_x^{0}(s)\bigr\rrvert ^2 \Bigr] \biggr).
\end{align*}
We estimate
\begin{align}
\label{ia-estimate} %
&\mathsf{E} \Bigl[\sup_{s\in[0,t\wedge\sigma^n_N]} \bigl
\llvert I_a^{n}(s)-I_a^{0}(s)
\bigr\rrvert^2 \Bigr]\le\mathsf{E} \Biggl[\sup_{s\in[0,t]}
\Biggl(\int_0^s \bigl\llvert
a^{n,n}_u-a^{0,0}_u \bigr
\rrvert1_u du \Biggr)^2 \Biggr]\notag
\\
&\quad\le\mathsf{E} \Biggl[ \Biggl(\int_0^t
\bigl\llvert a^{n,n}_u- a^{0,0}_u
\bigr\rrvert 1_u du \Biggr)^2 \Biggr] \le t \int
_0^t \mathsf{E} \bigl[ \bigl\llvert
a^{n,n}_u- a^{0,0}_u \bigr
\rrvert^2 1_u \bigr]du\notag
\\
&\quad\le C_t \int_0^t \bigl(
\mathsf{E} \bigl[ \bigl\llvert a^{n,n}_u-a^{n,0}_u
\bigr\rrvert^2 1_u \bigr] + \mathsf{E} \bigl[ \bigl
\llvert a^{n,0}_u-a^{0,0}_u \bigr
\rrvert^2 1_u \bigr] \bigr)du.
\end{align}
In turn,
\begin{align*}
&\int_0^t \mathsf{E} \bigl[ \bigl\llvert
a^{n,n}_u-a^{n,0}_u \bigr
\rrvert^2 1_u \bigr]du = \int_0^t
\! \mathsf{E} \bigl[ \bigl\llvert a^{n} \bigl(u,X^n(u)
\bigr)-a^{n} \bigl(u,X^0(u) \bigr) \bigr
\rrvert^2 1_u \bigr]du
\\
&\quad\le C_{N,t} \int_0^t \mathsf{E}
\bigl[ \bigl\llvert X^n(u)-X^n(0) \bigr
\rrvert^2 1_u \bigr]du \le C_{N,t} \int
_0^t \mathsf{E} \bigl[\varDelta^n(u)^2
1_u \bigr]du.
\end{align*}
By the Doob inequality and It\^o isometry we obtain
\begin{align*}
\mathsf{E} \Bigl[\sup_{s\in[0,t\wedge\sigma^n_N]} \bigl\llvert I_b^{n}(s)-I_b^{0}(s)
\bigr\rrvert^2 \Bigr] &\le C \mathsf{E} \bigl[ \bigl\llvert
I_b^{n} \bigl(t\wedge\sigma^n_N
\bigr)-I_b^{0} \bigl(t\wedge\sigma^n_N
\bigr) \bigr\rrvert^2 \bigr]
\\
&= C \int_0^t \mathsf{E} \bigl[ \bigl\llvert
b^{n,n}_s- b^{0,0}_s \bigr
\rrvert^21_s \bigr] ds.
\end{align*}
Estimating as in \eqref{ia-estimate}, we arrive at
\begin{align*}
&\int_0^t \mathsf{E} \bigl[ \bigl\llvert b^{n,n}_s - b^{0,0}_s \bigr\rrvert^21_s \bigr] ds\\
&\quad \le C_{N,t} \int_0^t \mathsf{E} \bigl[\varDelta^n(s)^2 1_s \bigr]ds+ C \int_0^t \mathsf{E} \bigl[ \bigl\llvert b^{n,0}_s-b^{0,0}_s \bigr\rrvert^2 1_s \bigr] ds.
\end{align*}
Finally, the Doob inequality yields
\begin{align*}
&\mathsf{E} \Bigl[\sup_{s\in[0,t\wedge\sigma^n_N]} \bigl\llvert I_c^{n}(s)-I_c^{0}(s)
\bigr\rrvert^2 \Bigr] \le C \mathsf{E} \bigl[ \bigl\llvert
I_c^{n} \bigl(t\wedge\sigma^n_N
\bigr)-I_c^{0} \bigl(t\wedge\sigma^n_N
\bigr) \bigr\rrvert^2 \bigr]
\\
&\quad= C \int_0^t \int_{\RR^m}
\mathsf{E} \bigl[ \bigl\llvert c^{n,n}_s(\theta) -
c^{0,0}_s( \theta ) \bigr\rrvert^2
1_s \bigr]\mu(d\theta) ds
\\
&\quad\le C\int_0^t \int_{\RR^m}
\bigl(\mathsf{E} \bigl[ \bigl\llvert c^{n,n}_s(\theta) -
c^{n,0}_s(\theta) \bigr\rrvert^2 1_s
\bigr]+ \mathsf{E} \bigl[ \bigl\llvert c^{n,0}_s(\theta) -
c^{0,0}_s(\theta) \bigr\rrvert^2 1_s
\bigr] \bigr)\mu(d\theta) ds.
\end{align*}
By (A2) we have
\begin{align*}
&C\int_0^t \int_{\RR^m}
\mathsf{E} \bigl[ \bigl\llvert c^{n,n}_s(\theta) -
c^{n,0}_s( \theta) \bigr\rrvert^2
1_s \bigr]\mu (d\theta) ds
\\
&\quad\le C_{N,t}\int_0^t \mathsf{E}
\bigl[ \bigl\llvert X^{n}(s) - X^{0}(s) \bigr
\rrvert^2 1_s \bigr] ds \le C_{N,t}\int
_0^t \mathsf{E} \bigl[\varDelta^{n}(s)^2
1_s \bigr] ds.
\end{align*}
Collecting all estimates, we arrive at the estimate
\begin{align*}
\mathsf{E} \bigl[\varDelta^n(t)^2 1_t
\bigr]&\le C \bigl\llvert X^n(0)-X^0(0) \bigr
\rrvert^2 + C_{N,t} \int_0^t
\mathsf{E} \bigl[\varDelta^n(s)1_s \bigr]ds
\\
&\quad+ C_t \int_0^t \mathsf{E}
\bigl[ \bigl\llvert\tilde a^{n,0}_s-\tilde
a^{0,0}_s \bigr\rrvert^2 1_s \bigr]
ds + C \int_0^t \mathsf{E} \bigl[ \bigl\llvert
b^{n,0}_s-b^{0,0}_s \bigr
\rrvert^2 1_s \bigr] ds
\\
&\quad+ C\int_0^t \int_{\RR^m}
\mathsf{E} \bigl[ \bigl\llvert c^{n,0}_s(\theta) -
c^{0,0}_s( \theta ) \bigr\rrvert^2
1_s \bigr]\mu(d\theta) ds,
\end{align*}
where we can assume without loss of generality that the constants are
nondecreasing in $t$.
The application of the Gronwall lemma leads to\vadjust{\eject}
\begin{align*}
&\mathsf{E} \bigl[\varDelta^n(T)^2 1_T
\bigr]
\\
&\quad\le C_{N,T} \Biggl( \bigl\llvert X^n(0)-X^0(0)
\bigr\rrvert^2 + \int_0^T
\mathsf{E} \bigl[ \bigl\llvert\tilde a^{n,0}_s-\tilde
a^{0,0}_s \bigr\rrvert^2 1_s \bigr]
ds
\\
&\qquad+ \int_0^T\! \mathsf{E} \bigl[ \bigl
\llvert b^{n,0}_s-b^{0,0}_s \bigr
\rrvert^2 1_s \bigr] ds +\! \int_0^T
\int_{\RR^m}\! \mathsf{E} \bigl[ \bigl\llvert
c^{n,0}_s( \theta) - c^{0,0}_s(\theta)
\bigr\rrvert^2 1_s \bigr]\mu(d \theta) ds \Biggr).
\end{align*}
We claim that the right-hand side of the latter inequality vanishes as
$n\to\infty$. Indeed, the integrands are bounded by $C_T(1+\llvert
X(s)\rrvert^2)$ due to (A1) and vanish pointwise due to (C1). Hence, the
convergence of integrals follows from the dominated convergence
theorem. The first term vanishes due to (C2); thus,
\[
\mathsf{E} \bigl[\varDelta^n(T)^2 1_T
\bigr]\to0,\quad n\to \infty.
\]

Now to prove the first statement, for any $\eps>0$, write
\begin{align*}
\pr \bigl(\varDelta^n(T)>\eps \bigr)&\le\frac{1}{\eps^2}\mathsf {E}
\bigl[\varDelta^n(T)^2 1_T \bigr] + \pr
\bigl( \sigma_N^n< T \bigr)
\\
&\le\frac{1}{\eps^2}\mathsf{E} \bigl[\varDelta^n(T)^2
1_T \bigr]+ \pr \Bigl(\sup_{t\in[0,T]} \bigl\llvert
X^n(0) \bigr\rrvert\ge N \Bigr)
\\
&\quad+ \pr \Bigl(\sup_{t\in[0,T]} \bigl\llvert X^0(0)
\bigr\rrvert\ge N \Bigr).
\end{align*}
This implies
\begin{gather*}
\varlimsup_{n\to\infty} \pr \bigl(\varDelta^n(T)>\eps \bigr)\le2
\sup_{n\ge
0}\, \pr \Bigl(\sup_{t\in[0,T]} \bigl\llvert
X^n(0) \bigr\rrvert\ge N \Bigr).
\end{gather*}
By the Chebyshev inequality we have
\begin{gather*}
\varlimsup_{n\to\infty} \pr \bigl(\varDelta^n(T)>\eps \bigr)\le
\frac
{2}{N^2}\sup_{n\ge0} \mathsf{E} \Bigl[\sup
_{t\in[0,T]} \bigl\llvert X^n(0) \bigr
\rrvert^2 \Bigr].
\end{gather*}
Therefore, using \eqref{momentbounded} and letting $N\to\infty$, we
get
\[
\varlimsup_{n\to\infty} P \bigl(\varDelta^n(T)>\eps \bigr)=0,
\]
as desired.

In order to prove the second statement, we repeat the previous
arguments with $\sigma_N^n \equiv T$, getting the estimate
\begin{align*}
\mathsf{E} \bigl[\varDelta^n(T)^2 \bigr]&\le
C_{T} \Biggl( \bigl\llvert X^n(0)-X^0(0) \bigr
\rrvert^2 + \int_0^T \mathsf{E}
\bigl[ \bigl\llvert\tilde a^{n,0}_s-\tilde
a^{0,0}_s \bigr\rrvert ^2 \bigr] ds
\\
&\quad+ \int_0^T \mathsf{E} \bigl[ \bigl
\llvert b^{n,0}_s-b^{0,0}_s \bigr
\rrvert^2 \bigr] ds
\\
&\quad+ \int_0^T \int_{\RR^m}
\mathsf{E} \bigl[ \bigl\llvert c^{n,0}_s(\theta) -
c^{0,0}_s( \theta ) \bigr\rrvert^2 \bigr]\mu(d
\theta) ds \Biggr).
\end{align*}
Hence, we get the required convergence as before, using the dominated
convergence theorem.
\end{proof}

\section{Convergence of hitting times}\label{sec4}

For each $n\ge0$, define the stopping time
%
\begin{equation}
\label{taun} \tau^{n} = \inf \bigl\{t\ge0: \varphi^n
\bigl(t,X^n(t) \bigr)\ge0 \bigr\}
\end{equation}
with the convention $\inf\varnothing= +\infty$; $\varphi^n$ is a
function satisfying certain assumptions to be specified later. In this
section, we study the convergence $\tau^n\to\tau^0$ as $n\to\infty$.

The motivation to study stopping times of the form \eqref{taun} comes
from the financial modeling. Specifically, let a financial market model
be driven by the process $X^n$ solving Eq.~\eqref{sde}, and $q>0$ be a
constant discount factor. Consider the problem of optimal exercise of
an American-type contingent claim with payoff function $f$ and maturity
$T$, that is,\ the maximization problem
\[
\mathsf{E} \bigl[e^{-q\tau}f \bigl(X^n(\tau) \bigr) \bigr]\to
\max,
\]
where $\tau$ is a stopping time taking values in $[0,T]$. Define the
value function
\[
v^n(t,x) = \sup_{\tau\in[t,T]}\mathsf{E}
\bigl[e^{-q(\tau-t)} f \bigl(X^n(\tau) \bigr) \mid
X^n(t)=x \bigr]
\]
as the maximal expected discounted payoff provided that the price
process $X^n$ starts from $x$ at the moment $t$; the supremum is taken
over all stopping times with values in $[t,T]$.

Then it is well known that the minimal optimal stopping time is given as
\[
\tau^{*,n} = \inf \bigl\{t\ge0: v^n \bigl(t,X^n(t)
\bigr) = f \bigl(X^n(t) \bigr) \bigr\},
\]
that is,\ it is the first time when the process $X^n$ hits the so-called
optimal stopping set
\[
\mathcal G^n = \bigl\{(t,x)\in[0,T]\times\RR^d:
v^n(t,x) = f(x) \bigr\}.
\]
Note that $\tau^{*,n} \le T$ since $v(T,x) = g(x)$. Since, obviously,
$v^n(t,x)\ge f(x)$, we may represent $\tau^{*,n}$ in the form \eqref
{taun} with $\varphi^n = f(x)-v^n(t,x)$.

\section{Convergence of hitting times for finite horizon}\label{sec5}
Let $T>0$ be a fixed number playing the role of finite maturity of an
American contingent claim. Let also the stopping times $\tau^n$, $n\ge
0$, be given by \eqref{taun} with $\varphi^n\colon[0,T]\times\RR
^d\to\RR$ satisfying the following assumptions.
\begin{itemize}
\item[(G1)] $\varphi^0\in C^{1}([0,T)\times\RR^d)$, and the
derivative $D_x\varphi^0$ is locally Lipschitz continuous in $x$, that
is,\ for all $t\in[0,T)$, $R>0$, $s\in[0,t]$, and $x,y\in B_d(R)$,
\[
\bigl\llvert D_x \varphi^0(s,x) - D_x
\varphi^0(s,y) \bigr\rrvert\le C_{t,R}\llvert x-y\rrvert.
\]
\item[(G2)] For all $n\ge0$ and $x\in\RR^d$, $\varphi^n(T,x) = 0$.
\item[(G3)] For all $t\in[0,T)$ and $x\in\RR^d$,
%
\begin{equation}
\label{diffusion>0} \bigl\llvert b^0(t,x)^\top D_x
\varphi^0(t,x) \bigr\rrvert> 0.
\end{equation}
\end{itemize}
Here by $b^0(t,x)^\top D_x \varphi^0(t,x)$ we denote the vector in
$\RR^k$ with $j$th coordinate equal to
\[
\sum_{i=1}^{d} b^0_{ij}(t,x)
\partial_{x_i}\varphi^0(t,x), \quad j=1,\dots,k.
\]

\begin{remark}\label{rem:G1G2}
Assumption \eqref{diffusion>0} means that the diffusion is acting
strongly enough toward the border of the set $\mathcal G^0_t :=  \{
x\in\RR^d: \varphi^0(t,x)\le0 \}$. In which situations does this
assumption hold, will be studied elsewhere. Here we just want to remark
that it is more delicate than it might seem. For example, consider the
optimal stopping problem described in the beginning of this section
with $n=0$ in \eqref{sde}. Then, under suitable assumptions (see,
e.g.,\ \cite{pham,zhang}), we have the smooth fit principle: $\partial
_x v^0(t,x) = \partial_x f(x)$ on the boundary of the optimal stopping
set. This means that we cannot set $\varphi^0(t,x) = f(x) - v^0(t,x)$
in order for \eqref{diffusion>0} to hold, contrary to what was
proposed in the beginning of the section.
\end{remark}

We will also assume the locally uniform convergence $\varphi^n\to
\varphi^0$.
\begin{itemize}
\item[(G4)] For all $t\in[0,T)$ and $R>0$,
\[
\sup_{(s,x)\in[0,t]\times B_d(R)} \bigl\llvert\varphi^n(s,x)- \varphi
^0(s,x) \bigr\rrvert\to0, \quad n\to\infty.
\]
\end{itemize}
\begin{remark}
The convergence of value functions in optimal stopping problems usually
holds under fairly mild assumptions on the convergence of coefficients
and payoffs. However, as we explained in Remark~\ref{rem:G1G2}, we
cannot use the value function for $\varphi^n$. This means that we
should find a function $\varphi^n$ defining $\mathcal G$ \textit
{different} from $v^n(t,x) - f(x)$, but it still should satisfy the
convergence assumption~(G4).

The question in which cases such functions exist and the convergence
assumption (G4) takes places will be a subject of our future research.
\end{remark}

In the case where $\nu$ has infinite activity, that is,\ $\mu(\RR
^m)=\infty$, we will also need some additional assumptions on the
components of Eq.~\eqref{sde}.
\begin{itemize}
\item[(A3)] For each $r>0$, $\mu(\RR^m \setminus B_m(r)
)<\infty$.
\item[(A4)] For all $t\ge0$, $x\in\RR^d$, and $\theta\in\RR^m$,
\[
\bigl\llvert c^0(t,x,\theta) \bigr\rrvert\le h(t,x)g(\theta),
\]
where the functions $g,h$ are locally bounded, $g(0)=0$, and $g(\theta
)\to0$, $\theta\to0$.
\end{itemize}
\begin{remark}\label{A34rem}
Assumption (A3) means that only small jumps of $\mu$ can accumulate on
a finite interval; assumption (A4) means that small jumps of $\mu$ are
translated by Eq.~\eqref{sde} to small jumps of $X^n$. An important
and natural example of a situation where these assumptions are
satisfied is an equation
\begin{align*}
X^0(t) &= X^0(0) + \int_0^t
a^0 \bigl(s,X^0(s) \bigr) ds + \int_0^t
b^0 \bigl(s,X^0(s) \bigr)dW(s)
\\
&\quad+ \int_{0}^{t} h^0
\bigl(s,X^0(s-) \bigr)dZ(s), \quad t\ge0,
\end{align*}
driven by a L\'evy process $Z(t) = \int_0^t \int_{\RR^m}\theta\,
\wnu(d\theta,ds)$.
\end{remark}

Now we are in a position to state the main result of this section.
\begin{theorem}\label{thm-convmoments}
Assume {(A1)--(A4), (C1), (C2), (G1)--(G4)}.
Then we have the following convergence in probability:
\[
\tau^{n}\overset{\pr} {\longrightarrow}\tau^{0}, \quad n
\to\infty.
\]
\end{theorem}
\begin{proof}

Let $\eps,\delta$ be small positive numbers. We are to show that for
all $n$ large enough,
%
\begin{equation}
\label{Pabs} \pr \bigl( \bigl\llvert\tau^{n}-\tau^{0}
\bigr\rrvert>\eps \bigr)<\delta.
\end{equation}
Using estimate \eqref{momentbounded} and the Chebyshev inequality, we
obtain that for some $R>0$,
\[
\pr \Bigl(\sup_{t\in[0,T]} \bigl\llvert X^0(t) \bigr
\rrvert\ge R \Bigr)< \frac{\delta}4.
\]

Denote $\mathcal K = [0,T-\eps/2]\times B_d(R+2)$,
\begin{align*}
M &= 1 + R + C_{T,R+2} + C_T + C_{T-\eps/2,R+2} + \sup
_{(t,x)\in
\mathcal K} \bigl( \bigl\llvert a(t,x) \bigr\rrvert+ \bigl\llvert
b(t,x) \bigr\rrvert
\\
&\quad+ \bigl\llvert\partial_t \varphi^0(t,x) \bigr
\rrvert+ \bigl\llvert D_x \varphi^0(t,x) \bigr\rrvert+
\bigl\llvert b^0(t,x)^\top D_x \varphi
^0(t,x) \bigr\rrvert^{-1} \bigr),
\end{align*}
where, with some abuse of notation, $C_{T,R+2}$ is the constant from
(A2) corresponding to $T$ and $R+2$, $C_T$ is the sum of constants from
(A1) and \eqref{sol-cont}, and $C_{T-\eps/2,R+2}$ is the constant
from (G1) corresponding to $T-\eps/2$ and $R+2$.

Let $\kappa\in(0,M]$ be a number, which we will specify later. Now we
claim that there exists a function $\varphi\in C^{1,2}([0,T)\times\RR
^d)$ such that
\[
\sup_{(t,x)\in\mathcal{K}} \bigl\llvert\varphi(t,x)- \varphi
^0(t,x) \bigr\rrvert< \varkappa/2
\]
and, moreover,
\begin{align*}
&\sup_{\substack{t\in[0,T-\eps/2]\\x\in B_d(R+1)}} \bigl( \bigl\llvert\partial_t
\varphi(t,x) \bigr\rrvert+ \bigl\llvert D_x \varphi(t,x) \bigr\rrvert
+ \bigl\llvert D^2_{xx}\varphi(t,x) \bigr\rrvert+ \bigl
\llvert b^0(t,x)^\top D_x \varphi(t,x) \bigr
\rrvert^{-1} \bigr)
\\
&\quad\le C_{T-\eps/2,R+2} + \sup_{(t,x)\in\mathcal K} \bigl( \bigl\llvert
\partial_t \varphi^0(t,x) \bigr\rrvert+ \bigl\llvert
D_x\varphi^0(t,x) \bigr\rrvert
\\
&\qquad+ \bigl\llvert b^0(t,x)^\top D_x
\varphi^0(t,x) \bigr\rrvert^{-1} \bigr)\\
&\quad \le M.
\end{align*}
Indeed, we can take the convolution $\varphi(t,x) = (\varphi
^0(t,\cdot)\star\psi)(x)$ with a delta-like smooth function $\psi$,
supported on a ball of radius less than $1$.

Further, by (G4) there exists $n_1\ge1$ such that for all $n\ge n_1$,
%
\begin{equation}
\sup_{(t,x)\in\mathcal{K}} \bigl\llvert\varphi^n(t,x)- \varphi
^0(t,x) \bigr\rrvert<\varkappa/2.
\end{equation}
On the other hand, by Theorem~\ref{ucp-thm} there exists $n_2\ge1$
such that for all $n\ge n_2$,
%
\begin{equation}
\pr \biggl(\sup_{t\in[0,T]} \bigl\llvert X^n(t)-X^0(t)
\bigr\rrvert\ge\frac{\varkappa
}{M} \biggr)<\frac{\delta}4.
\end{equation}
In what follows, we consider $n\ge n_1\vee n_2$.\vadjust{\eject}

Define the stopping time
\[
\sigma^n = \inf \biggl\{t\ge0: \bigl\llvert X^n(t)-X^0(t)
\bigr\rrvert\ge\frac
{\varkappa}{M}\ \text{ or } \bigl\llvert X^0(t)
\bigr\rrvert\ge R \biggr\} \wedge T.
\]
Write
\begin{align}
\label{Pabstau} \pr \bigl( \bigl\llvert\tau^{n}- \tau^{0}
\bigr\rrvert>\eps \bigr)&\le\pr \bigl( \bigl\llvert\tau^{n}-
\tau^{0} \bigr\rrvert>\eps, \sigma^n > T-\eps/2 \bigr)
\notag
\\
&\quad+ \pr \Bigl(\sup_{t\in[0,T]} \bigl\llvert X^0(t)
\bigr\rrvert\ge R \Bigr) + \pr \biggl(\sup_{t\in[0,T]} \bigl\llvert
X^n(t)-X^0(t) \bigr\rrvert\ge\frac{\varkappa}{M} \biggr)
\notag
\\
&< \pr \bigl( \bigl\llvert\tau^{n}-\tau^{0} \bigr\rrvert>
\eps, \sigma^n > T-\eps/2 \bigr) + \frac{\delta}2.
\end{align}
For any $t\le\sigma^n$,
\[
\bigl\llvert X^n(t) \bigr\rrvert\le \bigl\llvert X^0(t)
\bigr\rrvert+ \frac{\varkappa}M<R+1,
\]
and hence,
\begin{align*}
& \bigl\llvert\varphi^n \bigl(t,X^n(t) \bigr) - \varphi
\bigl(t,X^0(t) \bigr) \bigr\rrvert
\\
&\quad\le \bigl\llvert\varphi^n \bigl(t,X^n(t) \bigr) -
\varphi \bigl(t,X^n(t) \bigr) \bigr\rrvert+ \bigl\llvert\varphi
\bigl(t,X^n(t) \bigr) - \varphi \bigl(t,X^0(t) \bigr)
\bigr\rrvert
\\
&\quad\le\varkappa+ M \bigl\llvert X^n(t)-X^0(t) \bigr
\rrvert\le2\varkappa.
\end{align*}
Now take some $\eta\in(0,\eps/2]$ whose exact value will be
specified later and write the obvious inequality
\begin{align}
\label{Ptau*} &\pr \bigl(\tau_T^{*,0}+ \eps<
\tau_T^{*,n}, \sigma^n > T-\eps/2 \bigr)\notag
\\
&\quad\le\pr \bigl(\tau^{0}< T-\eps, \tau^{0}+\eta<
\tau^{n}, \sigma^n > T-\eps/2 \bigr).
\end{align}
Assume that $\tau^{0}< T-\eps$, $\tau^{0}+\eta<\tau^{n}$, $\sigma^n
> T-\eps/2$. Then, for all $t\in\break[\tau^{0},\tau^{0}+\eta]=:
\mathcal I_\eta$,
\[
\bigl\llvert\varphi^n \bigl(s,X^0(s) \bigr) - \varphi
\bigl(s,X^0(s) \bigr) \bigr\rrvert\le2\varkappa, \qquad
\varphi^n \bigl(t,X^n(t) \bigr)<0.
\]
Therefore, in view of the inequality $\varphi(\tau^{0},X^0(\tau
^{0}))\ge0$, we obtain
%
\begin{equation}
\label{infphi} \inf_{t\in\mathcal I_\eta} \varphi \bigl(t,X^0(t)
\bigr)\ge\varphi \bigl(\tau^{0},X^0 \bigl(
\tau^{0} \bigr) \bigr)-2\varkappa.
\end{equation}

Further, we will work with the expression $\varphi(t,X^0(t))- \varphi
(\tau^{0},X^0(\tau^{0}))$ for $t\in\mathcal{I}_\eta$. For
convenience, we will abbreviate $f_s = f(s,X^0(s))$; for example,
$\varphi_s = \break\varphi(s,X^0(s))$.

Let $r>0$ be a positive number, which we will specify later, and assume
that $\nu$ does not have jumps on $\mathcal I_\eta$ greater than $r$,
that is,\ $\nu((\RR^m\setminus B_m(r))\times\mathcal I_\eta)=0$.
Write, using the It\^o formula,
\begin{align*}
&\varphi \bigl(t,X^0(t) \bigr)- \varphi \bigl(\tau^{0}, X^0\bigl(\tau^{0} \bigr) \bigr)\\
&\quad = \int_{\tau^{0}}^t L_s \varphi_s ds + \int_{\tau^{0}}^t\bigl(D_x \varphi_s, b^0_s dW(s)\bigr)+ \int_{\tau_0}^t\!\int_{B_m(r)}\varDelta_s(\theta)\, \wnu(d\theta,ds)\\
&\quad =: I_1(t)+ I_2(t) + I_3(t),
\end{align*}\vadjust{\eject}
\!where
\begin{align*}
L_t \varphi_t &= \partial_t
\varphi_t + \bigl(D_x \varphi_t,a^0_t
\bigr) + \frac{1}2 \operatorname{tr} \bigl(b^0_t
\bigl(b^0_t \bigr)^\top D^2_{xx}
\varphi_t \bigr)
\\
&\quad+ \int_{B_m(r)} \bigl(\varDelta_s(\theta) -
\bigl(D_x \varphi_s, c^0
\bigl(s,X^0(s-),\theta \bigr) \bigr) \bigr)\mu(d\theta),
\\
\varDelta_s(\theta) &= \varphi \bigl(s,X^0(s-)+c
\bigl(s,X^0(s-),\theta \bigr) \bigr) - \varphi \bigl(s,X^0(s-)
\bigr).
\end{align*}

Start with estimating $I_1(t)$. Since $t\le\sigma^n\wedge(T-\eps
/2)$ for any $t\in\mathcal I_\eta$, by the definition of $M$ and
$\sigma^n$ we have
\begin{gather*}
\biggl\llvert\partial_t \varphi_t +
\bigl(D_x \varphi_t, a^0_t \bigr)
+ \frac{1}2 \operatorname{tr} \bigl(b^0_t
\bigl(b^0_t \bigr)^\top D^2_{xx}
\varphi_t \bigr) \biggr\rrvert\le M + M^2 +
M^3\le3M^3.
\end{gather*}
Further, by (A4), for $t\in\mathcal I_\eta$ and $\theta\in B_r$,
$\llvert c(t,X(t-),\theta)\rrvert\le
h(t,X(t-))g(\theta)\le K_1 m_r$, where $K_1 = \sup_{t\in[0,T],
\llvert x\rrvert\le R} h(t,x)$ and $m_r = \sup_{\theta\in B_m(r)}
g(\theta)$.
Since\break \mbox{$m_r\to0$}, $r\to0$, we can assume that $r$ is such that $m_r\le
1/K_1$. Then, for $t\in\mathcal I_\eta$, by the Taylor formula
\begin{align*}
& \biggl\llvert\int_{B_m(r)} \bigl(\varDelta_t(
\theta) - \bigl(D_x \varphi_t, c^0
\bigl(t,X^0(t-), \theta \bigr) \bigr) \bigr)\mu(d\theta) \biggr\rrvert
\\
&\quad\le\frac{1}2 \sup_{(u,x)\in[0,T]\times B_d(R+1)} \bigl\llvert
D^2_{xx}\varphi(u,x) \bigr\rrvert\int_{B_m(r)}
\bigl\llvert c \bigl(t,X^0(t-),\theta \bigr) \bigr
\rrvert^2 \mu(d\theta)
\\
&\quad\le\frac{1}2 M^2 \bigl(1+ \bigl\llvert X(t) \bigr
\rrvert^2 \bigr)\le\frac{1}2 M^2
\bigl(1+R^2 \bigr)\le M^4.
\end{align*}
Summing up the estimates, we get
%
\begin{equation}
\label{I1(t)} \bigl\llvert I_1(t) \bigr\rrvert\le \bigl(
3M^3 + M^4 \bigr)\eta\le4 M^4 \eta.
\end{equation}

Now proceed to $I_3(t)$. By the Doob inequality, for any $a>0$,
\begin{align*}
&\pr \Bigl(\sup_{t\in\mathcal I_\eta} \bigl\llvert I_3(t) \bigr
\rrvert\ge a,\sigma_n> T-\eps/2 \Bigr)\le\pr \Bigl(\sup
_{t\in[\tau_0,(\tau
_0+\eta)\wedge\sigma_n]} \bigl\llvert I_3(t) \bigr\rrvert\ge a \Bigr)
\\
&\quad\le C a^{-2}\mathsf{E} \Biggl[ \Biggl(\int_{0}^{T}
\int_{B_m(r)} \varDelta_s(\theta) 1_{[\tau_0,(\tau_0+\eta)\wedge
\sigma_n]}(s)
\wnu(d\theta,ds) \Biggr)^2 \Biggr]
\\
&\quad= C a^{-2}\int_{0}^{T} \int
_{B_m(r)} \mathsf{E} \bigl[\varDelta_s(
\theta)^2 1_{[\tau_0,(\tau
_0+\eta)\wedge\sigma_n]}(s) \bigr] \mu(d\theta)ds
\\
&\quad\le C a^{-2} M^2 \int_{0}^{T}
\int_{B_m(r)} \mathsf{E} \bigl[ \bigl\llvert c
\bigl(s,X^0(s-), \theta \bigr) \bigr\rrvert^2
1_{[\tau_0,(\tau_0+\eta )\wedge
\sigma_n]}(s) \bigr] \mu(d\theta)ds
\\
&\quad\le C a^{-2} M^2 \int_{0}^{T}
\int_{B_m(r)} \mathsf{E} \bigl[K_1^2m_r^2
1_{[\tau_0,(\tau_0+\eta
)\wedge\sigma_n]}(s) \bigr] \mu(d\theta)ds \le K_2 a^{-2}
m_r^2 \eta
\end{align*}
with some constant $K_2$. Further, we fix $a = \delta^2\eta^{1/2}$
and some $r>0$ such that $m_r^2\le\delta^5/(16K_2)$ and $m_r\le
1/K_1$. Then
\begin{gather*}
\pr \Bigl(\sup_{t\in\mathcal I_\eta} \bigl\llvert I_3(t) \bigr
\rrvert\ge\delta^2\eta^{1/2}, \sigma_n> T-
\eps/2 \Bigr)\le\frac{\delta}{16}.
\end{gather*}

Hence, in view of \eqref{Ptau*}--\eqref{I1(t)}, we obtain
\begin{align*}
&\pr \bigl(\tau^{0}+\eps<\tau^{n}, \sigma^n > T-\eps/2 \bigr)\\
&\quad\le\pr \Bigl(\inf_{t\in\mathcal I_\eta}I_2(t)\ge-2\varkappa-4M^4\eta- \delta^2\eta^{1/2},\sigma^n > T-\eps/2 \Bigr)\\
&\qquad+ \pr \Bigl(\sup_{t\in\mathcal I_\eta} \bigl\llvert I_3(t)\bigr\rrvert\ge\delta^2\eta^{1/2}, \sigma_n>T-\eps/2 \Bigr) + \pr \bigl(\nu \bigl( \bigl(\RR^m\setminus B_m(r) \bigr)\times\mathcal I_\eta \bigr) > 0 \bigr)\\
&\quad\le\pr \Bigl(\inf_{t\in\mathcal I_\eta}I_2(t)\ge-2\varkappa-4M^4\eta- \delta^2\eta^{1/2},\sigma^n > T-\eps/2 \Bigr)\\
&\qquad+ {\eta}\,\mu \bigl(\RR^m\setminus B_m(r) \bigr)+ \frac{\delta}{16}.
\end{align*}
Assume further that $\eta\le\eta_1:=\delta\, \mu(\RR^m)/16$ (not
yet fixing its exact value).
Setting $\varkappa= (\eta M^4)\wedge M$, we get
\begin{align}
\label{ptau*0} &\pr \bigl(\tau^{0}+\eps<\tau^{n},
\sigma^n > T-\eps/2 \bigr)\notag
\\
&\quad\le\pr \Bigl(\inf_{t\in\mathcal I_\eta}I_2(t)\ge-5\eta
M^4 -\delta^2\eta^{1/2},
\sigma^n >T-\eps/2 \Bigr) + \frac
{\delta}{8}.
\end{align}

Write $I_2(t) = J_1(t) + J_2(t) + J_3(t)$, where
\begin{align*}
J_1(t) &= \int_{\tau^{0}}^t
\bigl(D_x \varphi_s-D_x\varphi_{\tau
^{0}},
b^0_s\, dW(s) \bigr),
\\
J_2(t) &= \int_{\tau^{0}}^t
\bigl(D_x \varphi_{\tau^{0}}, \bigl(b^0_s-b^0_{\tau^{0}}
\bigr)dW(s) \bigr),
\\
J_3(t) &= \bigl(D_x \varphi_{\tau^{0}},b^0_{\tau^{0}}
\bigl(W(t)-W \bigl(\tau^{0} \bigr) \bigr) \bigr) =
\bigl(u_{\tau
^{0}},W(t)-W \bigl(\tau^{0} \bigr) \bigr);
\\
u_{s} &= b^0 \bigl(s,X^0(s)
\bigr)^\top D_x\varphi \bigl(s,X^0(s) \bigr).
\end{align*}

Taking into account that $(s,X^0(s))\in\mathcal K$ for $s\le\sigma
^n$, we estimate with the help of Doob's inequality
\begin{align*}
\mathsf{E} \Bigl[\sup_{t\in\mathcal I_\eta}J_1(t)^2
\mathbf {1}_{\sigma ^n>T-\epsilon/2} \Bigr]&\le\mathsf{E} \Bigl[\sup_{t\in[\tau^{0},(\tau^{0}+\eta )\wedge\sigma^n]}J_1(t)^2
\Bigr]
\\
&\le C\mathsf{E} \Biggl[ \Biggl(\int_{\tau^{0}}^{(\tau^{0}+\eta
)\wedge\sigma ^n}
\bigl(D_x \varphi_s-D_x\varphi_{\tau^{0}},
b^0_s\, dW(s) \bigr) \Biggr)^2 \Biggr]
\\
&\le C \mathsf{E} \Biggl[\int_{\tau^{0}}^{(\tau^{0}+\eta)\wedge
\sigma ^n}\llvert
D_x \varphi_s-D_x \varphi_{\tau^{0}}
\rrvert^2 \bigl\llvert b^0_s \bigr
\rrvert^2ds \Biggr]
\\
&\le C M^3 \mathsf{E} \Biggl[\int_{\tau^{0}}^{\tau^{0}+\eta}
\bigl\llvert X^0(s) - X^0 \bigl(\tau^{0}
\bigr) \bigr\rrvert^2ds \Biggr]
\\
&\le C M^4 \bigl(1+ \bigl\llvert X^0(0) \bigr
\rrvert^2 \bigr)\eta^2\le C M^4
\bigl(1+R^2 \bigr)\eta^2\le CM^6
\eta^2.
\end{align*}
Similarly, using (A2), we get
\[
\mathsf{E} \Bigl[\sup_{t\in\mathcal I_\eta}J_2(t)^2
\mathbf {1}_{\sigma ^n>T-\epsilon/2} \Bigr]\le C M^6 \eta^2.
\]
The Chebyshev inequality yields
\[
\pr \Bigl(\sup_{t\in\mathcal I_\eta} \bigl( \bigl\llvert J_1(t)
\bigr\rrvert+ \bigl\llvert J_2(t) \bigr\rrvert \bigr)\ge
\eta^{2/3}, \sigma^n>T-\eps/2 \Bigr)\le K_3
M^6\eta^{2/3}
\]
with certain constant $K_3$. Assume further that
\[
\eta\le\eta_2:= \biggl(\frac{\delta}{16 K_3 M^6} \biggr)^{3/2} ,
\]
in which case the right-hand side of the last inequality does not
exceed $\delta/16$, and that
\[
\eta\le\eta_3 := \frac{1}{125M^{9}},
\]
so that $\eta^{2/3}\ge5\eta M^3$.
Hence, in view of \eqref{ptau*0}, we obtain
\begin{align}
\label{ptau*01} &\pr \bigl(\tau^{0}+\eps<\tau^{n},
\sigma^n > T-\eps/2 \bigr)\notag
\\
&\quad\le\pr \Bigl(\inf_{t\in\mathcal I_\eta} J_3(t) \ge-5\eta
M^3 - \eta^{2/3}-\delta^2
\eta^{1/2},\sigma^n > T-\eps \Bigr)+ \frac{3\delta}{16}\notag
\\
&\quad\le\pr \Bigl(\inf_{t\in\mathcal I_\eta} J_3(t) \ge- 2\eta
^{2/3}-\delta^2\eta^{1/2}, \bigl(
\tau^{0},X^0 \bigl(\tau^{0} \bigr) \bigr)\in
\mathcal K \Bigr)+ \frac{3\delta}{16}.
\end{align}
Further, due to the strong Markov property of $W$,
\begin{align*}
&\pr \Bigl(\inf_{t\in\mathcal I_\eta} J_3(t) \ge- 2\eta
^{2/3}-\delta^2\eta^{1/2}, \bigl(
\tau^{0},X^0 \bigl(\tau^{0} \bigr) \bigr)\in
\mathcal K \Bigr)
\\
&\quad\! =\mathsf{E} \Bigl[\mathbf{1}_{\mathcal K} \bigl(\tau
^{0},X^0 \bigl(\tau^{0} \bigr) \bigr) \pr
\Bigl(\inf_{t\in
\mathcal I_\eta}J_3(t) \ge- 2\eta^{2/3}-
\delta^2\eta^{1/2} \mid F_{\tau^0} \Bigr) \Bigr]
\\
&\quad\! = \mathsf{E} \Bigl[\mathbf{1}_{\mathcal K} \bigl(\tau
^{0},X^0 \bigl(\tau^{0} \bigr) \bigr)
\\
&\qquad\! \times\pr \Bigl(\inf_{z\in[0,\eta]} \bigl(u(s,x),W(s\, {+}
\,z)\,{-}\,W(s) \bigr)\ge{-}\, 2\eta^{2/3}\,{-}\,\delta^2
\eta^{1/2} \Bigr)\,|_{(s,x)= (\tau^{0},X^0(\tau^{0}))} \Bigr],
\end{align*}
where $u(s,x) = b^0(s,x)^\top D_x\varphi(s,x)$. Observe now that $\{
(u(s,x),W(z+s)-W(s)),\break z\ge0\}$ is a standard Wiener process multiplied
by $\llvert u(s,x)\rrvert$. Therefore,
\begin{align*}
&\pr \Bigl(\inf_{z\in[0,\eta]} \bigl(u(s,x),W(s+z)-W(s) \bigr) \ge-2
\eta^{2/3} -\delta^2\eta^{1/2} \Bigr)
\\
&\quad= 1- 2\pr \bigl( \bigl(u(s,x),W(s+\eta)-W(s) \bigr)< -2\eta
^{2/3} -\delta^2\eta^{1/2} \bigr)
\\
&\quad=1- 2\varPhi \biggl( -\frac{2\eta^{2/3}+\varDelta^2\eta
^{1/2}}{\llvert u(s,x)\rrvert\eta^{1/2} } \biggr) = 1- 2\varPhi \biggl( -
\frac{2\eta^{1/6}+\delta^2}{\llvert u(s,x)\rrvert} \biggr),
\end{align*}
where $\varPhi$ is the standard normal distribution function.
Thus,\vadjust{\eject}
\begin{align*}
&\pr \Bigl(\inf_{t\in\mathcal I_\eta} J_3(t) \ge- 2\eta
^{2/3}-\delta^2\eta^{1/2}, \bigl(
\tau^{0},X^0 \bigl(\tau^{0} \bigr) \bigr)\in
\mathcal K \Bigr)
\\
&\quad\le\mathsf{E} \biggl[\mathbf{1}_{\mathcal K} \bigl(\tau
^{0},X^0 \bigl(\tau^{0} \bigr) \bigr) \biggl(1-
2\varPhi \biggl( - \frac{2\eta ^{1/6}+\delta^2}{\llvert u(\tau^{0},X^0(\tau
^{0}))\rrvert } \biggr) \biggr) \biggr]
\\
&\quad\le1- 2\varPhi \bigl( -M \bigl(2\eta^{1/6}+\delta^2
\bigr) \bigr)\le\frac{M\sqrt2}{\sqrt{\pi}} \bigl(2\eta ^{1/6}+\delta
^2 \bigr).
\end{align*}
Note that the definition of $M$ does not depend on $\delta$. Thus, we
can assume without loss of generality that $\delta\le\sqrt{\pi
}/(32M\sqrt{2})$. Finally, if
\[
\eta\le\eta_4 := \biggl(\frac{\delta\sqrt{\pi}}{64\sqrt
{2}} \biggr)^6,
\]
then
%
\begin{equation}
\label{PinfJ3} \pr \Bigl(\inf_{t\in\mathcal I_\eta} J_3(t) \ge- 2
\eta^{2/3}-\varDelta^2\eta^{-1/2}, \bigl(
\tau^{0},X^0 \bigl(\tau^{0} \bigr) \bigr)\in
\mathcal K \Bigr)\le\frac{\delta}{16}.
\end{equation}
Now we can fix $\eta= \min\{\eps/2,\eta_1,\eta_2,\eta_3,\eta_4\}
$, making all previous estimates to hold. Combining \eqref{ptau*01}
with \eqref{PinfJ3}, we arrive at
\begin{gather*}
\pr \bigl(\tau^{0}+\eps<\tau^{n}, \sigma^n >
T-\eps/2 \bigr)\le\frac{\delta}{4}.
\end{gather*}
Similarly,
\begin{gather*}
\pr \bigl(\tau^{n}+\eps<\tau^{0}, \sigma^n >
T-\eps/2 \bigr)\le\frac{\delta}{4},
\end{gather*}
and hence
\begin{gather*}
\pr \bigl( \bigl\llvert\tau^{n}-\tau^{0} \bigr\rrvert>
\eps, \sigma^n > T-\eps/2 \bigr)\le\frac{\delta}{2}.
\end{gather*}
Plugging this estimate into \eqref{Pabstau}, we arrive at the desired
inequality \eqref{Pabs}.
\end{proof}
\begin{remark}
It is easy to modify the proof for the case where \eqref{diffusion>0}
holds for all $(t,x)\in\mathcal G^0:= \{(t,x)\in[0,T)\times\RR
^d: \varphi(t,x)=0 \}$. Indeed, the continuity would imply that
\eqref{diffusion>0} holds in some neighborhood of $\mathcal G^0$, which is
sufficient for the argument.
\end{remark}
\begin{remark}
As we have already mentioned, assumptions (A3) and (A4) are not needed
in the case $\mu(\RR^m)<\infty$. Indeed, we can set $r=0$ in the
previous argument and skip the estimation of $I_3(t)$. Nevertheless,
these assumptions does not seem very restrictive, as we pointed out in
Remark~\ref{A34rem}.
\end{remark}

\subsection{Convergence of hitting times for infinite horizon}\label{sec5.1}
Here we extend the results of the previous subsection to the case of
infinite time horizon.
Let, as before, the stopping times $\tau^n$, $n\ge0$, be given by
\eqref{taun}. We impose the following assumptions.
\begin{itemize}
\item[(H1)] $\varphi^0\in C^{1}([0,\infty)\times\RR^d)$, and
$D_x\varphi^0$ is locally Lipschitz continuous in $x$, that is,\ for
all $T>0$, $R>0$, $t\in[0,T]$, and $x,y\in B_d(R)$,
\[
\bigl\llvert D_x \varphi^0(t,x) - D_x
\varphi^0(t,y) \bigr\rrvert\le C_{T,R}\llvert x-y\rrvert.
\]
\item[(H2)] $\tau^0 <\infty$ a.s.
\item[(H3)] For all $t\ge0$ and $x\in\RR^d$,
\begin{equation*}
\bigl\llvert D_x \varphi^0(t,x)b^0(t,x)^\top
\bigr\rrvert> 0.
\end{equation*}
\item[(H4)] For all $t\ge0$ and $R>0$,
\[
\sup_{(s,x)\in[0,t]\times B_d(R)} \bigl\llvert\varphi^n(t,x)- \varphi
^0(t,x) \bigr\rrvert\to0, \quad n\to\infty.
\]
\end{itemize}
\begin{theorem}\label{thm-con-infty}
Assume {(A1), (A2), (C1), (C2)}, {(H1)--(H4)}.
Then we have the following convergence in probability:
\[
\tau^{n}\overset{\pr} {\longrightarrow}\tau^{0}, \quad n
\to\infty.
\]
\end{theorem}
\begin{proof}
Fix arbitrary $\eps\in(0,1)$ and $\delta>0$.
Since $\tau^{0}<\infty$ a.s.,
$\pr(\tau^{0}>T-1)\le\delta$ for some $T>1$. For $n\ge0$, $t\in
[0,T]$, and $x\in\RR^d$, define \ $\tilde\varphi^n(t,x) = \varphi
^n(t,x)\mathbf{1}_{[0,T)}(t)$, $\tau^{n}_T = \tau^{n}\wedge T$. Then
the functions $\tilde\varphi^n$, $n\ge0$, satisfy (G1)--(G3) and $
\tau_T^n = \inf \{t\ge0: \tilde\varphi^n(t,X^n(t))\ge0
\}$.
Therefore, in view of Theorem~\ref{thm-convmoments},
\[
\pr \bigl( \bigl|\tau^{n}_T-\tau^{*,0}_T
\bigr|> \eps \bigr)\to0, \quad n\to\infty.
\]
We estimate
\begin{align*}
\pr \bigl( \bigl\llvert\tau^{n}-\tau^{0} \bigr\rrvert>
\eps \bigr)&\le\pr \bigl( \bigl\llvert\tau^{n}_T-
\tau^{0}_T \bigr\rrvert>\eps \bigr) + \pr \bigl(
\tau^{0}>T-1 \bigr)
\\
&\le\pr \bigl( \bigl\llvert\tau^{n}_T-
\tau^{0}_T \bigr\rrvert>\eps \bigr) + \delta.
\end{align*}
Hence,
\begin{gather*}
\varlimsup_{n\to\infty} \pr \bigl( \bigl\llvert\tau^{n}- \tau
^{0} \bigr\rrvert> \eps \bigr)\le\delta.
\end{gather*}
Letting $\delta\to0$, we arrive at the desired convergence.
\end{proof}

\begin{example}
Let $d=k=m=1$ and for all $t\ge0$, $x,\theta\in\RR$, $a^n(t,x) =
a^n$, $b^n(t,x)= b^n$, $c^n(t,x,\theta) = c^n \theta$, where
$a^n,b^n,c^n\in\RR$.
Then we have a sequence of L\'evy processes
\[
X^n(t) = X^n(0) + a^n t + b^n
W(t) + c^n \int_0^t \int
_{\RR}\theta\, \wnu(ds,d\theta).
\]
Consider the following times:
\[
\tau^{n} = \inf \bigl\{t\ge0: X^n(t) \ge
h^n(t) \bigr\}\wedge T, \quad n\ge0,
\]
of crossing some curve $h\in C^1([0,T))$.

Assume that $a^n\to a^0$, $b^n\to b^0\neq0$, $c^n\to c^0$, and
$X^n(0)\to X^0(0)$ as $n\to\infty$ and, for any $t\in[0,T)$, $\sup_{s\in[0,t]} \llvert h^n(t)-h^0(t)\rrvert\to0$ as $n\to\infty$.
Then $\tau
^{n} \overset{\pr}{\longrightarrow} \tau^{0}$, $n\to\infty$.
Indeed, setting $\varphi^n(t,x) = (h^n(t)-x)\mathbf{1}_{[0,T)}(t)$,
we can check that all assumptions of Theorem~\ref{thm-convmoments} are
in force.
\end{example}
\begin{example}
Let $d=k=m=1$. Suppose that the coefficients $a^n$, $b^n$, $c^n$
satisfy (A1), (A2) and that the convergence (C1)--(C3) takes place.
Assume that $b^0(t,x)>0$ for all $t\ge0$ and $x\in\RR$. Define
\[
\tau^{n} = \inf \bigl\{t\ge0: X^n(t)\notin
\bigl(l^n,r^n \bigr) \bigr\}, \quad n\ge0.
\]
It is not hard to check that, due to the nondegeneracy of $b^0$, $\tau
^0<\infty$ a.s. Assume that $l^n\to l^0$, $r^n\to r^0$, $n\to\infty$.
Then, setting $\varphi^n(t,x) = (x-l^n)(r^n-x)$ and using Theorem~\ref{thm-con-infty},
we get the convergence $\tau^{n} \overset{\pr}{\longrightarrow} \tau^{0}$, $n\to\infty$.
\end{example}

\section*{Acknowledgments}

The author would like to thank the anonymous referee whose remarks led
to a substantial improvement of the manuscript.

%

%
\end{document}